\pdfoutput=1
\pdfoutput=1
\pdfoutput=1
\pdfoutput=1
\pdfoutput=1
\documentclass[a4paper,12 pt]{article}
\usepackage{calc}
\usepackage[all]{xy}
\usepackage[centertags]{amsmath}
\usepackage{latexsym}
\usepackage{amsfonts}
\usepackage{graphicx}
\usepackage{tikz}
\usepackage{cases}
\usepackage{amssymb}
\usepackage{amsthm}
\usepackage{color}
\usepackage{fancyhdr}
\usepackage[dvips]{epsfig}
\usepackage{newlfont}
\usepackage[latin1]{inputenc}
\usepackage[latin1]{inputenc}
\usepackage[english,french]{babel}
\usepackage{graphicx}
\usepackage{t1enc}
\usepackage[english,french]{babel}
\usepackage{fancybox}
\usepackage{graphicx}
\usepackage{t1enc}
\usepackage[french2]{minitoc}
\usepackage{mathrsfs}
\usepackage{amsfonts}
\usepackage{wasysym}
\usepackage{hyperref}
\allowdisplaybreaks
\usepackage{float}
\usepackage{authblk}
\usepackage{geometry}

\fancypagestyle{plain}{ \fancyhead{}

	\rhead{\textbf{}} \rfoot{\footnotesize{\textsf{\tiny }}}
	\cfoot{\footnotesize{\textbf{ \thepage}}}} \hfuzz2pt
\newlength{\defbaselineskip}
\numberwithin{equation}{section} 
\newtheorem{theorem}{Theorem}[section]
\newtheorem{corollary}[theorem]{Corollary}

\newtheorem{lemma}[theorem]{Lemma}
\newtheorem{example}{Example}[section]
\newtheorem{proposition}[theorem]{Proposition}

\newtheorem{definition}[theorem]{Definition}
\newtheorem{remark}{Remark}[section]

\newcommand{\A}{{ \mathcal{A}}}
\newcommand{\Pa}{{ \mathcal{P}}}

 \makeatletter
\newcommand{\thechapterwords}
{ \ifcase \thechapter\or 1\or 2\or 3\or 4\or 5\or
	6\or 7\or 8\or 9\or 10\or 11\fi}
\def\thickhrulefill{\leavevmode \leaders \hrule height 2ex \hfill \kern \z@}
\def\@makechapterhead#1{%
	\vspace*{15\p@}%
	{\parindent \z@ \centering \reset@font
		\thickhrulefill\quad
		\scshape  {\chapnumfont \@chapapp{}}{\chapnumfont \thechapterwords}
		\quad \thickhrulefill
		\par\nobreak
		\vspace*{15\p@}%
		\interlinepenalty\@M
		\hrule
		\vspace*{15\p@}%
		\huge {\bfseries  #1}\par\nobreak
		\par
		\vspace*{15\p@}%
		\hrule
		\vskip 15\p@
	}}
	\def\@makeschapterhead#1{%
		\vspace*{15\p@}%
		{\parindent \z@ \centering \reset@font
			\thickhrulefill
			\par\nobreak
			\vspace*{15\p@}%
			\interlinepenalty\@M
			\hrule
			\vspace*{15\p@}%
			\Huge \bfseries #1\par\nobreak
			\par
			\vspace*{15\p@}%
			\hrule
			\vskip 30\p@
		}}
		
		\DeclareFixedFont{\chapnumfont}{T1}{phv}{b}{n}{20pt}
		\DeclareFixedFont{\chapchapfont}{T1}{phv}{b}{n}{16pt}
		\DeclareFixedFont{\chaptitfont}{T1}{phv}{b}{n}{24.88pt}
		\def\@makechapterhead#1{%
			\vspace*{15\p@}%
			{\parindent \z@ \centering \reset@font
				\thickhrulefill\quad
				\scshape {\chaptitfont\color[rgb]{0.00,0.50,1.00}\@chapapp{}}
				{\chapnumfont \thechapterwords}
				\quad \thickhrulefill
				\par\nobreak
				\vspace*{15\p@}%
				\interlinepenalty\@M
				\hrule
				\vspace*{15\p@}%
				{\Large\bfseries #1}\par\nobreak
				\par
				\vspace*{15\p@}%
				\hrule
				\vskip 30\p@
			}}%
			
\begin{document}
				\title{Diffeomorphism Group Actions on Para-Kähler Structures: Lifts and Induced Dynamics on 3-Webs} 
				\author[1]{Bertuel Tangue Ndawa}
				\author[2]{Ferdinand Ngakeu}
				\author[3]{Nasser Saipele Nansidi}
				\author[4]{Thomas Bouetou Bouetou}
				\affil[1]{
					University of Ngaoundere, Cameroon.}
				\affil[2]{
					University of Douala, Cameroon.}
				\affil[3]{
					University of Maroua, Cameroon.}
				\affil[4]{University of Yaounde, Cameroon}

				\date{ }

				\maketitle
				
				\selectlanguage{english}
\begin{center}
	To Professor Jo\"el Tossa
\end{center}

	\section*{Abstract}
Let $M$ be a smooth manifold equipped with a para-K\"ahler structure. We define
an action of the diffeomorphism group $\operatorname{Diff}(M)$ on the space of
para-K\"ahler structures and prove that this action preserves the associated
affine structures. Although this result is well known, we provide a proof based
on this action, showing in particular that an isometry commutes with the
Levi--Civita connection.
We then construct several natural liftings of this action to the tangent bundle
$TM$, the cotangent bundle $T^*M$, and the Whitney sum
$E=TM\oplus T^*M$. In certain cases, these lifted actions induce natural
dynamical systems on a distinguished class of $3$-webs.\vspace{0.25cm}

				\textbf{Keywords}: Connexion, distribution, isometry, metric, paracomplex, para-Kähler, $3$-web.\vspace{0.25cm}
				
				\textbf{MSC2010}: 53D05, 53D12.
				
				\textbf{Acknowledgment}:




\bigskip

\section{Introduction}
An almost paracomplex (product) structure on a $2n$-dimensional manifold $M$, with $n>0$, is a $(1,1)$-tensor field $F$ satisfying $F^2=\operatorname{Id}$, and such that the eigendistributions $F^\pm$ associated with the eigenvalues $\pm 1$ both have dimension $n$. If $g$ is a neutral metric on $M$ that is anti-invariant under $F$, equivalently, if $F$ is an anti-isometry of $g$ ($g$-compatible), namely
$g(FX,FY)=-g(X,Y)$
for all vector fields $X,Y$ on $M$, then the pair $(g,F)$ is called an almost para-Kähler (para-Hermitian) structure. If, in addition, the Levi-Civita connection $\nabla=\nabla^g$ of $g$ satisfies $\nabla F=0$, then $(g,F)$ is called a para-Kähler structure, and the triple $(M,F,g)$ is called a para-Kähler manifold.

Para-Kähler structures have been extensively studied in the literature. One of the main reasons for this interest is that such structures correspond bijectively to bi-Lagrangian structures, namely triples $(\omega,\mathcal{F}_1,\mathcal{F}_2)$, where $\omega$ is a symplectic form and $\mathcal{F}_1,\mathcal{F}_2$ are two transverse Lagrangian foliations defined by $F^\pm$. The tensors $g$, $\omega$, and $F$ are related by
$g(\cdot,\cdot)=\omega(F\cdot,\cdot).$
Consequently, any two of these tensors uniquely determine the remaining one.
A para-Kähler structure also induces a deformed Lie bracket on vector fields, defined by
$[X,Y]_F=[FX,Y]+[X,FY]-F[X,Y]$
for all vector fields $X,Y\in\mathfrak{X}(M)$. This bracket gives rise to an associated cohomology theory, which is isomorphic to the de Rham cohomology via $F$.

In \cite{TNB2}, the author studies a lift of the action of the symplectic
group to the space of bi-Lagrangian structures, under the assumption that the
underlying manifold is parallelizable. The main result is based on a lifting
procedure on $TM$ and $T^*M$ for bi-Lagrangian manifolds. The action considered
there is induced by the map $(\psi,X)\longmapsto \psi_*X;$
that is, by the natural action of the diffeomorphism group on the space of
vector fields. This construction is further extended in \cite{TNB5} to Whitney
sums $E=TM\oplus T^*M$ in the setting of parallelizable manifolds. Moreover, in \cite{TNB4}, the author introduces several distinct lifts of the same action to $TM$ and $T^*M$.

In the present work, we adopt an approach different from those developed in
the aforementioned papers. More precisely, by studying the natural action of
the diffeomorphism group on the space of para-K\"ahler structures, we obtain
a unified framework that generalizes the previous constructions. In particular,
the results of \cite{TNB2,TNB4,TNB5} may be recovered as special cases of the
present setting. Furthermore, the lifting of para-K\"ahler structures from $M$ to $TM$ and $T^*M$ induces,
in certain cases, dynamics on an associated class of $3$-webs.

Before explaining our results more precisely and proving them, we first recall some definitions, fix the notation, and state the known results that will be needed throughout the paper.

\section{Notation and Preliminaries}\label{sub1}


\subsection{General Conventions}

Throughout this paper, all objects are assumed to be smooth unless otherwise
stated. Let $k\in\mathbb{N}^*$. For simplicity, we write $[k]$ for the set
$[k]=\{1,2,\ldots,k\}.$

We use the Einstein summation convention: whenever an index appears once as an
upper index and once as a lower index in a product, summation over its
corresponding range is understood. For instance,
$$
\lambda^i\xi_i=\sum_{i=1}^n \lambda^i\xi_i,
\quad
X^i\partial_{y^i}
=
X^i\frac{\partial}{\partial y^i}
=
\sum_{i=1}^n X^i\frac{\partial}{\partial y^i}.
$$

For $k\in\mathbb{N}$, the notation $I_k$ denotes the identity matrix of size
$k$, while $\operatorname{Id}_{S}$, or simply $\operatorname{Id}$ when no
confusion can arise, denotes the identity map on a set $S$.

The matrix representation of a linear map $f$ with respect to a given basis
will be denoted by $(f)$.

\subsection{Symplectic manifolds}

Symplectic manifolds have been studied since the eighteenth century. Among the fundamental examples in symplectic geometry, the cotangent bundle $T^*M$ carries a natural symplectic structure induced by
the tautological one-form. In this section, we recall this construction and
state its functoriality under cotangent lifts.

Let $M$ be an $m$-dimensional smooth manifold, and let
${}^*\hspace{-0.1cm}\pi:T^*M\longrightarrow M$ denote the canonical projection. The tautological
one-form, also called the Liouville one-form, is the one-form $\theta$ on
$T^*M$ defined by
$$
\theta_{(x,\alpha_x)}(v)
=
\alpha_x\bigl(T_{(x,\alpha_x)}{}^*\hspace{-0.1cm}\pi(v)\bigr),
\quad
(x,\alpha_x)\in T^*M,\quad
v\in T_{(x,\alpha_x)}T^*M.
$$
Its exterior derivative $
\omega_{\mathrm{can}}:=d\theta$
is called the canonical symplectic form on $T^*M$, or equivalently the
Liouville two-form.

 Let $(U,x^1,\dots,x^m)$ be a local coordinate chart on the manifold $M$, and let
$(T^*U,x^1,\dots,x^m,\xi_1,\dots,\xi_m)$ be the induced cotangent coordinate
chart on $T^*M$. Then the tautological one-form and the canonical symplectic
form are locally given by
\begin{equation}\label{theta_et_dtheta_LAP}
	\theta=\sum_{i=1}^{m}\xi_i\,dx^i,
\quad
\omega_{\mathrm{can}}=d\theta=\sum_{i=1}^{m}d\xi_i\wedge dx^i.
\end{equation}

\begin{proposition}\label{prop:cotangent_lift_symplectomorphism}
	Let $\psi:M_1\longrightarrow M_2$ be a diffeomorphism. Its cotangent lift is
	the map
	$$
	\widehat{\psi}:T^*M_1\longrightarrow T^*M_2,
	\quad
	(x,\alpha_x)\longmapsto
	\bigl(\psi(x),(\psi^{-1})^*\alpha_x\bigr).
	$$
	Then $\widehat{\psi}$ is a symplectomorphism from $(T^*M_1,d\theta_1)$ to
	$(T^*M_2,d\theta_2)$, where $d\theta_1$ and $d\theta_2$ denote the canonical
	symplectic forms on $T^*M_1$ and $T^*M_2$, respectively.
	
The cotangent lifts $\widehat{\operatorname{Diff}}(M)$ of diffeomorphisms of the manifold $M$ form a subgroup of
$\operatorname{Diff}(T^*M)$. More precisely, the set
$$
\widehat{\operatorname{Diff}}(M)
:=
\left\{
\widehat{\psi}:T^*M\longrightarrow T^*M
\; ; \;
\psi\in \operatorname{Diff}(M)
\right\}
$$
is a subgroup of $\operatorname{Diff}(T^*M)$.
\end{proposition}

\subsection{Linear Connections}
We denote by $\operatorname{Conn}(M)$ the set of all linear connections on $M$.

Let $\nabla\in \operatorname{Conn}(M)$. The torsion tensor of $\nabla$, denoted by $T_{\nabla}$, or simply by $T$ when there is no ambiguity, is defined by
\begin{equation*}
	T_{\nabla}(X,Y)
	=
	\nabla_XY-\nabla_YX-[X,Y],
	\quad X,Y\in\mathfrak{X}(M).
\end{equation*}

The curvature tensor of $\nabla$, denoted by $R_{\nabla}$, or simply by $R$ when there is no ambiguity, is defined by
\begin{equation*}
	R_{\nabla}(X,Y)Z
	=
	\nabla_X\nabla_YZ
	-
	\nabla_Y\nabla_XZ
	-
	\nabla_{[X,Y]}Z,
	\quad
	X,Y,Z\in\mathfrak{X}(M).
\end{equation*}

We say that a linear connection $\nabla$:
\begin{enumerate}
	\item preserves a symplectic form $\omega$ if $\nabla\omega=0$; that is,
	\begin{equation*}
		X\bigl(\omega(Y,Z)\bigr)
		=
		\omega(\nabla_XY,Z)
		+
		\omega(Y,\nabla_XZ),
		\quad X,Y,Z\in\mathfrak{X}(M);
	\end{equation*}
	
	\item preserves a distribution $\mathcal{D}$ if
	$\nabla\Gamma(\mathcal{D})\subseteq \Gamma(\mathcal{D})$, meaning that
	$\nabla_XY\in\Gamma(\mathcal{D})$
	for all $X\in\mathfrak{X}(M)$ and all $Y\in\Gamma(\mathcal{D})$.
\end{enumerate}


\subsection{Para-K\"ahler Structures}\label{PKsubsec1}
Throughout this work, $M$ denotes a smooth manifold of dimension $2n$, and $\operatorname{Met}_n=\operatorname{Met}_n(M)$ denotes the set of pseudo-Riemannian metrics on $M$ of neutral signature $(n,n)$. Let $g \in \operatorname{Met}_n$. We denote by $\operatorname{Iso}_g$ the
isometry group of $g$, namely $\operatorname{Iso}_g
=
\left\{
\psi \in \operatorname{Diff}(M) \;:\; \psi^*g=g
\right\}.$

An almost product structure on $M$ is a tensor field
$F\in \Gamma(\operatorname{End}(TM))$ satisfying $F^2=\operatorname{Id}$ and $F\neq \operatorname{Id}$.
The eigenbundles associated with the eigenvalues $+1$ and $-1$ are denoted by
$$
F^+:=\ker(F-\operatorname{Id}),
\quad
F^-:=\ker(F+\operatorname{Id}).
$$

If
\begin{equation*}
	\operatorname{rank}F^{+}=\operatorname{rank}F^{-}=n,
\end{equation*}
then $F$ is called an almost paracomplex (par-Hermitian) structure. In some cases, we will identify $F$ with the pair $(F^{-},F^{+})$. Moreover, viewing $F$ as a $C^\infty$-linear map from $\mathfrak{X}(M)$ to itself, it follows that $\mathfrak{X}(M)=F^{-}\oplus F^{+}$.

We denote by $\mathcal{P}c=\mathcal{P}c(M)$ the set of almost paracomplex structures on $M$:
$$
\mathcal{P}c(M)
=
\left\{
F\in \Gamma(\operatorname{End}(TM))
\;\middle|\;
F^2=\operatorname{Id},
\quad
\operatorname{rank}F^+=\operatorname{rank}F^-=n
\right\}.
$$

If, in addition, $F$ is compatible with the metric $g$ in the sense that
$$
g(FX,FY)=-g(X,Y),
\quad
\forall X,Y\in \mathfrak{X}(M),
$$
then $F$ is said to be $g$-compatible. Equivalently, $F$ is an anti-isometry
of $g$.

In local coordinates, this condition is expressed by
$$
(F)^t(g)(F)=-(g).
$$

We denote by $\mathcal{P}c_g=\mathcal{P}c(M,g)$ the set of $g$-compatible almost paracomplex structures:

\begin{equation*}
	\mathcal{P}c_g
=
\left\{
F\in \mathcal{P}c(M)
\;\middle|\;g(FX,FY)=-g(X,Y)
	\quad \forall X,Y\in \mathfrak{X}(M)
\right\}.
\end{equation*}

For each $F\in \mathcal{P}c_g$, the associated fundamental $2$-form
$\omega$ is defined by
$$
\omega(X,Y):=g(FX,Y),
\quad
X,Y\in \mathfrak{X}(M).
$$
Equivalently, since $F^2=\operatorname{Id}$, one obtains
$$
g(X,Y)=\omega(FX,Y),
\quad
X,Y\in \mathfrak{X}(M).
$$
Thus, any two of the tensors $g$, $F$, and $\omega$ determine the third. In
this sense, these objects may be regarded as mutually associated, and we write
formally
$$
g=(\omega,F),
\quad
F=(g,\omega),
\quad
\omega=(g,F).
$$

Locally, the associated matrices satisfy
$$
(\omega)=(F)^t(g),
\quad
(g)=(F)^t(\omega).
$$

Since both $g$ and $\omega$ are non-degenerate, each tensor
$\eta\in\{g,\omega\}$ induces a vector bundle isomorphism $
\eta^\flat:TM\longrightarrow T^*M,
\quad
(x,X_x)\longmapsto \bigl(x,\eta_x(X_x,\cdot)\bigr).
$
We denote the inverse isomorphism by
$\eta^\sharp:T^*M\longrightarrow TM.$

We denote by $\nabla^g$, or simply by $\nabla$ when no confusion can arise,
the Levi-Civita connection associated with the pseudo-Riemannian metric $g$.
It is the unique torsion-free linear connection which is compatible with $g$;
that is, $\nabla g=0.$
In local coordinates $(x^1,\dots,x^{2n})$, its Christoffel symbols are given by
\begin{equation*}
	\Gamma_{ij}^{k}
	=
	\frac{1}{2}g^{k\ell}
	\left(
	\frac{\partial g_{j\ell}}{\partial x^{i}}
	+
	\frac{\partial g_{i\ell}}{\partial x^{j}}
	-
	\frac{\partial g_{ij}}{\partial x^{\ell}}
	\right),
\end{equation*}
where $(g^{k\ell})$ denotes the inverse matrix of $(g_{k\ell})$.

Since $\omega$ is nondegenerate, it is symplectic if and only if $d\omega=0$.

 The almost para-Hermitian structure $(g,F)$ is called a para-K\"ahler structure if
$$
\nabla^g F=0.
$$
Equivalently, $\nabla^g\omega=0$. Moreover, this condition is equivalent to the simultaneous vanishing of $d\omega$ and of the Nijenhuis tensor $N_F$ of $F$, where
$$
N_F(X,Y)=[FX,FY]-F[FX,Y]-F[X,FY]+F^2[X,Y],
\quad X,Y\in\mathfrak{X}(M).
$$
In this case, the triple $(M,g,F)$ is called a para-K\"{a}hler manifold.

We denote by
\begin{equation*}
\mathcal{P}k=	\mathcal{P}k (M)
	:=
	\{(g,F)\;:\; g\in \operatorname{Met}_n(M),\ F\in \mathcal{P}c_g(M),\ \nabla F=0\}
\end{equation*}
the set of para-K\"ahler structures on $M$.

A para-K\"ahler structure $(g,F)\in \Pa k(M)$ is said to be affine if the Levi-Civita connection $\nabla^g$ is flat. Such a structure is characterized by the following result, which follows directly from \cite[Theorem 2]{Hess}.

\begin{remark}\label{LAP_Aff-Struc_Rem}
A para-K\"ahler structure $(g,F)$ on $M$ is said to be affine if and only if,
for every point $x\in M$, there exists a local coordinate system $(p^1,\ldots,p^n,q^1,\ldots,q^n)$ defined on a neighbourhood of $x$ such that
$$
F^+=\left\langle\frac{\partial}{\partial p^1},\ldots,\frac{\partial}{\partial p^n}
\right\rangle_{C^\infty(M)},\quad
F^-=\left\langle\frac{\partial}{\partial q^1},\ldots,\frac{\partial}{\partial q^n}
\right\rangle_{C^\infty(M)}.
$$
Moreover, in these coordinates, the symplectic form $\omega$ and the metric $g$ are given respectively by
$$\omega=\sum_{i=1}^ndq^i\wedge dp^i,\quad
g
=
-\sum_{i=1}^n
\left(
dp^i\otimes dq^i+dq^i\otimes dp^i
\right),
$$
and all Christoffel coefficients of the Levi-Civita connection $\nabla^g$
vanish.

Equivalently, in such a coordinate chart, the matrices of the fundamental
$2$-form $\omega$, the paracomplex structure $F$, and the metric tensor $g$
are respectively given by
\begin{equation}\label{LAP_w-g-F_Cano}
	(\omega)
	=
	\begin{pmatrix}
		0&-I_n\\
		I_n&0
	\end{pmatrix},
	\quad
	(F)
	=
	\begin{pmatrix}
		I_n&0\\
		0&-I_n
	\end{pmatrix},
	\quad
	(g)
	=
	\begin{pmatrix}
		0&-I_n\\
		-I_n&0
	\end{pmatrix}.
\end{equation}
\end{remark}



\subsection{Prolongations of geometric objects to the tangent bundle}

We recall in this subsection some basic facts concerning complete, vertical, and
horizontal lifts to the tangent bundle. For further details on complete and
vertical lifts, we refer the reader to \cite{YKI,YKII}, and for horizontal
lifts to \cite{Dom}. If $\zeta$ is a geometric object on a manifold $M$, we
denote by $\zeta^v$, $\zeta^c$, and $\zeta^h$ its vertical, complete, and
horizontal lifts to $TM$, respectively.

Let $\pi:TM\to M$ be the natural projection. For
$f\in C^\infty(M)$, its vertical, horizontal, and complete lifts are defined by
$$
f^v=f\circ\pi= f^h,
\quad
f^c=df,
$$
where $df$ is viewed as a fiberwise linear function on $TM$.

For $X\in\mathfrak{X}(M)$, write locally
$$
X=X^i\frac{\partial}{\partial x^i},
\quad
\nabla_{\frac{\partial}{\partial x^i}}
\frac{\partial}{\partial x^j}
=
\Gamma^k_{ij}
\frac{\partial}{\partial x^k}.
$$
With respect to the induced coordinates
$(x^1,\ldots,x^n,y^1,\ldots,y^n)$ on $TM$, one has
\begin{equation}
	\begin{aligned}
		X^v
		&=
		X^i\frac{\partial}{\partial y^i},\\
		X^h
		&=
		X^i\frac{\partial}{\partial x^i}
		-
		X^i\Gamma^k_{ij}y^j
		\frac{\partial}{\partial y^k},\\
		X^c
		&=
		X^i\frac{\partial}{\partial x^i}
		+
		\frac{\partial X^i}{\partial x^j}y^j
		\frac{\partial}{\partial y^i}.
	\end{aligned}
	\label{eqlift5}
\end{equation}

Moreover, for all $X,Y\in\mathfrak{X}(M)$, one has
\begin{equation}
	\begin{aligned}[l]
		&[X^c,Y^c]=[X,Y]^c,\quad[X^v,Y^v]=0,\quad[X^c,Y^v]=[X,Y]^v,\\
		&[X^h,Y^v]=(\nabla_XY)^v,\hspace{0.5cm}[X^h,Y^h] =[X,Y]^h-\bigl(R(X,Y)y\bigr)^v,\\
		&[X^c,Y^h]
		=
		[X,Y]^h
		-
		\left((\mathcal L_X\nabla)(Y,y)\right)^v,
	\end{aligned}
	\label{eqlift4}
\end{equation}
where $R$ is the curvature tensor of $\nabla$, $\mathcal L_X\nabla$ denotes the
Lie derivative of the connection $\nabla$ with respect to $X$, and
$y\in T_xM$ denotes the tangent vector corresponding to a point $(x,y)\in TM$.

\begin{remark}\label{DPDrem1}
	Let $\mathcal{D}$ be a distribution on $M$. For $a,b\in\{c,h,v\}$,  
	 $d\in\{c,h\}$, and , $\alpha,\beta\in\mathbb{R}\setminus\{0\}$
	we define the distributions $\mathcal{D}^{ab}$ and $\mathcal{D}^{v(d)}_{\alpha,\beta}$
 on $TM$ by
	\begin{align}
		\Gamma(\mathcal{D}^{ab})
		&=
		\left\langle X^a,\;X^b;\;X\in\Gamma(\mathcal{D})\right\rangle_
		{C^\infty(TM)}.
		\label{lift of foliation ab}\vspace{0.25cm}\\
		\Gamma\left(\mathcal{D}^{v(d)}_{\alpha,\beta}\right) &=
		\left\langle X^{v(d)}:=\alpha X^v+\beta X^d;\;X\in\Gamma(\mathcal{D})\right\rangle_
		{C^\infty(TM)}\label{lift of foliation v(d)}
.
	\end{align}
	If $\mathcal{D}^a$ and $\mathcal{D}^b$ are transverse, then their
	intersection is trivial, and hence
	\begin{equation} 
		\Gamma\bigl(\mathcal{D}^{ab}\bigr) 
		= 
		\Gamma\bigl(\mathcal{D}^a\bigr) 
		\oplus
		\Gamma\bigl(\mathcal{D}^b\bigr). 
		\label{eqlift7}
	\end{equation}
	
	Assume now that $\mathcal{D}$ is involutive. Then, by the bracket
	relations \eqref{eqlift4}, the distributions $\mathcal{D}^c$,
	$\mathcal{D}^v$, and  $\mathcal{D}^{cv}$ are involutive. 
	
	For fixed constants $\alpha,\beta$ satisfying $\alpha\beta\neq0$, the
	distribution $\mathcal{D}^{v(c)}_{\alpha,\beta}$ is involutive when
	$\mathcal{D}$ is the trivial distribution.
	
	The involutivity of distributions involving horizontal lifts depends, in
	general, on the connection $\nabla$. For example, $\mathcal{D}^{hv}$ is
	involutive if and only if
	$$
	\nabla_XY\in\Gamma(\mathcal{D})
	\quad\text{and}\quad
	R(X,Y)y\in\mathcal{D}
	$$
	for all $X,Y\in\Gamma(\mathcal{D})$ and all $y\in TM$, where the second
	condition is understood fiberwise, namely $R(X,Y)y\in\mathcal{D}_{\pi(y)}$.
	
	Similarly, the involutivity of $\mathcal{D}^{ch}$ imposes additional
	conditions involving both the curvature of $\nabla$ and the tensor
	$\mathcal L_X\nabla$. In particular, if $\nabla$ is flat along
	$\mathcal{D}$ and the vector fields in $\Gamma(\mathcal{D})$ are affine
	with respect to $\nabla$; that is,
	$$
	\mathcal L_X\nabla=0,
	\quad X\in\Gamma(\mathcal{D}),
	$$
	then the bracket relations simplify accordingly.
\end{remark}

\subsection{Conormal prolongations of distributions to the cotangent bundle}
Let $\mathcal{D}$ be a distribution on an $m$-dimensional manifold $M$. The
conormal bundle associated with $\mathcal{D}$ is defined as the annihilator of
$\mathcal{D}$ in $T^*M$, namely
\begin{equation*}
	\begin{aligned}
		N^*\mathcal{D}&=
		\{(x,\alpha_x)\in T^*M:\ \alpha_x(v)=0,\ \forall v\in \mathcal{D}_x,\ x\in M\}\\ &=
		\bigcup_{x\in M}\{x\}\times N_x^*\mathcal{D},
	\end{aligned}
\end{equation*}
where
\begin{equation*}
	N_x^*\mathcal{D}
	=
	\{\alpha_x\in T_x^*M:\ \alpha_x(v)=0,\ \forall v\in \mathcal{D}_x\}.
\end{equation*}
Equivalently, its space of smooth sections is given by
\begin{equation*}
	\Gamma(N^*\mathcal{D})
	=
	\{\alpha\in\Omega^1(M):\ \alpha(X)=0,\ \forall X\in\Gamma(\mathcal{D})\}.
\end{equation*}

\begin{remark}
	If $\mathcal{D}$ has constant rank $r$, then $N^*\mathcal{D}$ is a vector
	subbundle of $T^*M$ of rank $
		\operatorname{rank}(N^*\mathcal{D})=m-r.$		
	We define the associated subbundle
	$
	N^{*t}\mathcal{D}:=\mathcal{D}\oplus N^*\mathcal{D},
	$
	whose space of smooth sections is
	$
	\Gamma(N^{*t}\mathcal{D}^N)
	=
	\Gamma(\mathcal{D})\oplus \Gamma(N^*\mathcal{D}).
	$
\end{remark}



\section{Statements and proofs of results}
\subsection{Action of diffeomorphism group on the set of paracomplex structure}
We begin this section with the following observation.
\begin{remark}\label{act}
	Let 
	$\psi,\varphi:M\longrightarrow M$ be two diffeomorphisms. Recall that
	$(\psi\circ\varphi)_*=\psi_*\circ\varphi_*$.
	
	For every connection $\nabla\in \operatorname{Conn}(M)$ and every diffeomorphism
	$\psi$ of $M$,  $\nabla^\psi$ defined  by 
	$$	(\nabla^\psi)_X Y
	=
	\psi_*\left(
	\nabla_{\psi_*^{-1}X}\,\psi_*^{-1}Y
	\right),
	\quad X,Y\in\mathfrak{X}(M)$$  is a connection. 	
	Moreover, for all diffeomorphisms $\psi,\varphi:M\to M$, one has
	$\nabla^{\psi\circ\varphi}=(\nabla^\varphi)^\psi.$
\end{remark}

The first result concerns the push-forward of para-Kahler structures. The precise statement is as follows:

\begin{theorem}\label{LAPactionTheo}
	Let $(M,g,F)$ be a para-K\"ahler manifold. Define the map
	\begin{equation}\label{LAPactMap}
		\mathcal{A}: \operatorname{Diff}(M)\times\mathcal{P}k \longrightarrow \mathcal{P}k,\quad
		(\psi,(g,F))\longmapsto \bigl((\psi^*)^{-1}g,\; F^\psi := \psi_* \circ F \circ \psi_*^{-1}\bigr).
	\end{equation}
	Then $\mathcal{A}$ is a left action. Moreover,  $\nabla^{(\psi^*)^{-1}g}$ is the pushed-forward connection $\nabla^{g}$. As a consequence, the restriction of $\mathcal{A}$ to affine par-Kahler is well defined.
\end{theorem}

\begin{proof}
	\begin{lemma}\label{LAPpushforparLem}	Let $(M,g,F)$ be a para-K\"ahler manifold. Let $N$ be a manifold diffeomorphic to
		$M$. Then, for any diffeomorphism $\psi:M\longrightarrow N$, the 	
		pair
		$\bigl((\psi^*)^{-1}g,\; F^\psi \bigr)$
		defines a para-K\"ahler structure on $N$. 		
		Moreover,   one has $\nabla=\nabla^g$ if and only if  $\nabla^{(\psi^*)^{-1}g} =\nabla^\psi$.
		Furthermore, the structure
		$(g,F)$ is affine if and only if the structure $\bigl((\psi^*)^{-1}g,\; F^\psi := \psi_* \circ F \circ \psi_*^{-1}\bigr)$
		is affine.
	\end{lemma}
	\begin{proof}
Set $g_N := (\psi^{-1})^*g$. Since $\psi^{-1}:N\to M$ is a diffeomorphism and $g$ is a neutral metric, the form $g_N$ is  a neutral metric as well.

We now show that $F^\psi$ is an almost paracomplex structure.

\begin{equation*}
	\begin{aligned}
		F^\psi \circ F^\psi
		&= \psi_* \circ F \circ \psi_*^{-1} \circ \psi_* \circ F \circ \psi_*^{-1} \\
		&= \psi_* \circ F^2 \circ \psi_*^{-1} \\
		&= \mathrm{Id}.
	\end{aligned}
\end{equation*}

Moreover, since $\psi$ is a diffeomorphism, there exist unique distributions $(F_N^{-},F_N^{+})$ such that
\begin{equation*}
	\psi_*(F^-)=F_N^-,
	\quad
	\psi_*(F^+)=F_N^+,
	\quad
	\operatorname{rank} F_N^-=\operatorname{rank} F_N^+=n.
\end{equation*}
Observe that
\begin{equation*}
	F^\psi\big|_{F_N^{+}}=\mathrm{Id},
	\quad
	F^\psi\big|_{F_N^{-}}=-\mathrm{Id}.
\end{equation*}
Hence, $F^\psi$ is a paracomplex structure.

Furthermore, for every $X,Y\in\mathfrak{X}(N)$, we compute
\begin{align*}
	(\psi^{-1})^*g\big(F^\psi(X),F^\psi(Y)\big)
	&= g\big(\psi^{-1}_*F^\psi(X),\psi^{-1}_*F^\psi(Y)\big) \\
	&= g\big(F(\psi^{-1}_*X),F(\psi^{-1}_*Y)\big) \\
	&= -\,g\big(\psi^{-1}_*X,\psi^{-1}_*Y\big) \\
	&= -(\psi^{-1})^*g(X,Y).
\end{align*}
In the above, we use that $F$ is $g$-compatible. Therefore, $F^\psi$ is $(\psi^{-1})^*g$-compatible.

\medskip
Action on the Levi--Civita connection.
Let $\nabla^{g}=\nabla$. Recall that the Levi--Civita connection of a pseudo-Riemannian metric on a manifold is the unique torsion-free connection that preserves the metric.

First, for all $X,Y\in\mathfrak{X}(N)$, one has
\begin{equation*}
	T_{\nabla^\psi}(X,Y)
	=
	\psi_*\Bigl(
	T_\nabla(\psi_*^{-1}X,\psi_*^{-1}Y)
	\Bigr).
\end{equation*}
Thus,
\begin{equation}\label{LAPtornul}
	T_\nabla=0
	\Longleftrightarrow
	T_{\nabla^\psi}=0.
\end{equation}

Next, for $X,Y,Z\in\mathfrak{X}(N)$, we obtain
\begin{equation*}
	(\nabla^\psi_X g_N)(Y,Z)
	=
	(\nabla_{\psi_*^{-1}X}g)\big(\psi_*^{-1}Y,\psi_*^{-1}Z\big).
\end{equation*}
Therefore,
\begin{equation}\label{LAPprg}
	\nabla g=0
	\Longleftrightarrow
	\nabla^\psi g_N=0.
\end{equation}

Moreover,
\begin{align*}
	\nabla^\psi F^\psi
	&= \psi_* \circ \nabla_{\psi_*^{-1}}\Bigl(\psi_*^{-1}\circ\bigl(\psi_* \circ F \circ \psi_*^{-1}\bigr)\Bigr) \\
	&= \psi_* \circ \nabla_{\psi_*^{-1}X}\bigl(F(\psi_*^{-1})\bigr) \\
	&= \psi_*\bigl(\nabla_{\psi_*^{-1}X}F\bigr)\circ \psi_*^{-1}.
\end{align*}

Since $\psi$ is a diffeomorphism, this implies that
\begin{equation}\label{LAPprF}
	\nabla F=0
	\Longleftrightarrow
	\nabla^\psi F^\psi=0.
\end{equation}

By combining \eqref{LAPtornul}, \eqref{LAPprg}, and \eqref{LAPprF}, it follows that
\begin{equation}\label{LAPequiconnex}
	\nabla=\nabla^{g}
	\Longleftrightarrow
	\nabla^\psi=\nabla^{(\psi^*)^{-1}g}.
\end{equation}

Finally, the curvature of $\nabla^\psi$ satisfies
\begin{equation*}
	R_{\nabla^\psi}(X,Y)Z
	=
	\psi_*\Bigl(
	R_\nabla(\psi_*^{-1}X,\psi_*^{-1}Y)\psi_*^{-1}Z
	\Bigr).
\end{equation*}
Therefore,  
\begin{equation}\label{LAPCurvNul}
R_\nabla=0\Longleftrightarrow R_{\nabla^\psi}=0.
\end{equation}

Consequently, the para-K\"ahler structure
$(g,F)$ is affine if and only if 
the structure	
$\bigl((\psi^*)^{-1}g,\; F^\psi := \psi_* \circ F \circ \psi_*^{-1}\bigr)$
is affine.
	\end{proof}

Now, we can easily prove Theorem~\ref{LAPactionTheo}.

By Lemma~\ref{LAPpushforparLem}, $\mathcal{A}$ is well defined.
The action properties of $\mathcal{A}$ follow from the push-forward action of $\operatorname{Diff}(M)$ on $\mathfrak{X}(M)$.
Moreover, by \eqref{LAPCurvNul}, the affine condition is preserved under this action.
This proves Theorem~\ref{LAPactionTheo}.
\end{proof}

\begin{remark} Let $(g,F)$ be a para-K\"ahler structure on a manifold $M$.
	If $\psi \in \operatorname{Iso}_g$, then
	$\mathcal{A}(\psi,(g,F)) = (g,F^\psi)$ (see \eqref{LAPactMap}).
	Therefore, by Theorem~\ref{LAPactionTheo} or Lemma~\ref{LAPpushforparLem},
	we conclude that $(\nabla^g)^{\psi} = \nabla^g$.
	Equivalently, $\psi$ commutes with the Levi--Civita connection $\nabla^g$.
	This fact is well known in the literature.
\end{remark}

\subsection{Prolongation of the action $\A$}
This section is devoted to the construction of lifted para-K\"ahler structures
on the tangent bundle $TM$ and the cotangent bundle $T^*M$, and, more
generally, on the Whitney sum $E=TM\oplus T^*M.$
It will then follow that the lifted action introduced in
Theorem~\ref{LAPactionTheo} preserves these structures. This section therefore
constitutes the core of the construction announced in the title of the present
work. Before formulating the main statement precisely, we recall the following
definition.

\begin{definition}
	Let $(g,F)$ be a para-K\"ahler structure on $M$, and let $\omega$ denote its
	associated symplectic form.
	
	On $T^*M$, we define
	$$
	N^{*t}F=\bigl(N^{*t}F^{-},N^{*t}F^{+}\bigr),
	\quad
	N^{*t}g=(d\theta,N^{*t}F),
	\quad
	\widetilde{g}_N=({}^*\hspace{-0.1cm}\pi^*\omega+d\theta,N^{*t}F).
	$$
	
	On $TM$, for each $ab\in\{cv,ch,hv\}$, we define
	$$
	F^{ab}=\bigl((F^{-})^{ab},(F^{+})^{ab}\bigr),
	\quad
	g^{ab}=\bigl((\omega^*)^{-1}d\theta,F^{ab}\bigr).
	$$
\end{definition}

\begin{theorem}\label{LAPLiftTheo}
	Let $(g,F)$ be a para-K\"ahler structure on $M$, and let $\omega$ denote
	the associated symplectic form. Then the following assertions hold:
	\begin{enumerate}
		\item The pair $(N^{*t}g,N^{*t}F)$ defines an affine para-K\"ahler structure
		on $T^*M$.
		
		\item The pair $(\widetilde{g}_N,N^{*t}F)$ defines a para-K\"ahler structure
		on $T^*M$. Moreover, this structure is affine if and only if $(g,F)$ is
		affine.
		
		\item We suppose that  $(g,F)$ is affine.
		\begin{enumerate}
\item For each $ab\in\{cv,hv\}$, the pair
$(g^{ab},F^{ab})$ defines a para-K\"ahler structure on $TM$, where the
horizontal lift is taken with respect to the Levi-Civita connection
$\nabla^g$. 
\item The pairs $\left( g^{a}, F^{a} \right)$, $a\in\{h,v\}$,
		induce  a collection of  3 webs $TM$.
	\end{enumerate}
	\end{enumerate}
\end{theorem}

\begin{proof}
	Let $(g,F)$ be a para-K\"ahler structure on $M$, and let
$
(U,p^1,\dots,p^n,q^1,\dots,q^n)
$
be a coordinate chart adapted to $(F^-,F^+)$. Thus,
locally, $$
F^+
=
\left\langle
\frac{\partial}{\partial p^1},\dots,
\frac{\partial}{\partial p^n}
\right\rangle,
\quad
F^-
=
\left\langle
\frac{\partial}{\partial q^1},\dots,
\frac{\partial}{\partial q^n}
\right\rangle.
$$

Let $(T^*U,
p^1,\dots,p^n,q^1,\dots,q^n,\xi_1,\dots,\xi_{2n})$ denote the induced cotangent bundle chart on \(T^*U\). With respect to these
coordinates, the lifted distributions admit the following local expressions:
\begin{equation}\label{LAP_F-N}
	\begin{cases}
		N^{*t}F^+
		=
		\left\langle
		\frac{\partial}{\partial p^1},\dots,
		\frac{\partial}{\partial p^n},
		\frac{\partial}{\partial \xi_{n+1}},\dots,
		\frac{\partial}{\partial \xi_{2n}}
		\right\rangle,
		\\[0.25cm]
		N^{*t}F^-
		=
		\left\langle
		\frac{\partial}{\partial q^1},\dots,
		\frac{\partial}{\partial q^n},
		\frac{\partial}{\partial \xi_1},\dots,
		\frac{\partial}{\partial \xi_n}
		\right\rangle,
	\end{cases}
\end{equation}
and 
	\begin{equation}\label{LAP_dtheta}
	d\theta
	=
	\sum_{i=1}^n d\xi_i\wedge dp^i
	+
	\sum_{i=1}^n d\xi_{n+i}\wedge dq^i.
\end{equation}
\paragraph{The pair $(N^{*t}g,N^{*t}F)$ is an affine para-K\"ahler structure on $T^*M$:}
By combining equalities~\eqref{LAP_F-N} and~\eqref{LAP_dtheta} with
Remark~\ref{LAP_Aff-Struc_Rem}, the assertion follows.

\paragraph{The pair $(\widetilde{g}_N,N^{*t}F)$ is a para-K\"ahler structure on
	$T^*M$.}

Observe first that
\begin{equation}\label{LAP_tildeOmega}
	\widetilde{\omega}
	=
	\pi^*\omega+d\theta
	=
	\sum_{1\leq i<j\leq 2n}(\omega_{ij}\circ\pi)\,dx^i\wedge dx^j
	+
	\sum_{i=1}^n d\xi_i\wedge dx^i
	+
	\sum_{i=1}^n d\xi_{n+i}\wedge dx^{n+i},
\end{equation}
where
$$
(x^1,\ldots,x^{2n})
:=
(p^1,\ldots,p^n,q^1,\ldots,q^n),
\quad
(\omega_{ij})=(\omega).
$$

By combining \eqref{LAP_tildeOmega} with \eqref{LAP_F-N}, we obtain
\begin{equation}\label{LAP_Omtilde-F_N}
	(\widetilde{\omega})
	=
	\begin{pmatrix}
		(\omega)&-I_{2n}\\
		I_{2n}&0
	\end{pmatrix},
	\quad
	(N^{*t}F)
	=\begin{pmatrix}
		(F)&0\\
		0&-(F)
	\end{pmatrix}, \quad (F)=\begin{pmatrix}
	I_n&0\\
	0&-I_n
	\end{pmatrix}.
\end{equation}

Since locally $(g)=(F)^t(\omega)$, it follows from
\eqref{LAP_Omtilde-F_N} that the matrix of $N^{*t}\tilde{g}$ is given by
\begin{equation}\label{LAP_g_N}
	(N^{*t}\tilde{g})
	=
	\begin{pmatrix}
		(\omega)_{\mathrm{sym}\text{-}s}&-(F)\\
	-(F)&0
	\end{pmatrix}
	=
	\begin{pmatrix}
		\bigl((F)^t(g)\bigr)_{\mathrm{sym}\text{-}s}&-(F)\\
		-(F)&0
	\end{pmatrix},
\end{equation}
where $(\omega)_{\mathrm{sym}\text{-}s}$ denotes the symmetric $2n\times 2n$
matrix determined by the diagonal entries of $(\omega)$ together with its
entries above the diagonal.

Moreover, from \eqref{LAP_Omtilde-F_N}, one has
$$
\widetilde{\omega}(N^{*t}FX,N^{*t}FY)
=
-\widetilde{\omega}(X,Y),
\quad
X,Y\in\mathfrak{X}(T^*M).
$$
Consequently,
\begin{equation}\label{LAP_Fg-com}
	\widetilde{g}_N(N^{*t}FX,N^{*t}FY)
	=
	-\widetilde{g}_N(X,Y),
	\quad
	X,Y\in\mathfrak{X}(T^*M).
\end{equation}

By \eqref{LAP_F-N}, the endomorphism $N^{*t}F$ is completely integrable. Combining
this fact with \eqref{LAP_Fg-com}, we conclude that
$(N^{*t}\tilde{g},N^{*t}F)$ is a para-K\"ahler structure on $T^*M$.

Finally, by \eqref{LAP_g_N}, the Christoffel coefficients of the Levi-Civita
connection associated with $N^{*t}\tilde{g}$ are obtained from those of $g$ by
the corresponding block construction. Hence $(N^{*t}\tilde{g},N^{*t}F)$ is affine
if and only if $(g,F)$ is affine.

\paragraph{We assume that  $(g,F)$ is an affine  para-K\"ahler structure on $M$.\\}
Let $(p^i,q^i)=(p^1,\ldots,p^n,q^1,\ldots,q^n)$ be an adapted coordinates to $(g,F)$. 
From  Remark~\ref{LAP_Aff-Struc_Rem} and  the second identity in \eqref{theta_et_dtheta_LAP}, we get
\begin{equation}\label{F_for_lift_LAP}
	F^+
	=
	\left\langle
	\frac{\partial}{\partial p^1},\ldots,
	\frac{\partial}{\partial p^n}
	\right\rangle,
	\quad
	F^-
	=
	\left\langle
	\frac{\partial}{\partial q^1},\ldots,
	\frac{\partial}{\partial q^n}
	\right\rangle,
\end{equation}
and
\begin{equation}\label{LAF_dtheta_TM}
	\left(\omega^\flat\right)^*(d\theta)
	=
	\sum_{i=1}^n dy^{n+i}\wedge dp^i
	-
	\sum_{i=1}^n dy^{i}\wedge dq^i
\end{equation}
where $(p^1,\ldots,p^n,q^1,\ldots,q^n,y^1,\ldots,y^{2n})$
are the tangent bundle coordinates on $TM$ induced by the local coordinates $(p^1,\ldots,p^n,q^1,\ldots,q^n)$.

Moreover, by combining  \eqref{eqlift5} and \eqref{F_for_lift_LAP}, we obtain
\begin{equation}\label{LAP_F-et-F-ab}
\begin{array}{cl}
			\left(F^+\right)^v
			=
			\left\langle
			\frac{\partial}{\partial y^1},\ldots,
			\frac{\partial}{\partial y^n}
			\right\rangle,
			&\hspace{0.5cm}
			\left(F^-\right)^v
			=
			\left\langle
			\frac{\partial}{\partial y^{n+1}},\ldots,
			\frac{\partial}{\partial y^{2n}}
			\right\rangle,
			\\[0.5cm]
			\left(F^+\right)^h
			=
			\left\langle
			\frac{\partial}{\partial p^1},\ldots,
			\frac{\partial}{\partial p^n}
			\right\rangle,
			&\hspace{0.5cm}
			\left(F^-\right)^h
			=
			\left\langle
			\frac{\partial}{\partial q^1},\ldots,
			\frac{\partial}{\partial q^n}
			\right\rangle,
			\\[0.5cm]
			\left(F^+\right)^c
			=
			\left\langle
			\frac{\partial}{\partial p^1}
			,
			\ldots,
			\frac{\partial}{\partial p^n}
			
			\right\rangle,
			&\hspace{0.5cm}
			\left(F^-\right)^c
			=
			\left\langle
			\frac{\partial}{\partial q^1}
			,
			\ldots,
			\frac{\partial}{\partial q^n}
			\right\rangle,\\[0.5cm]
		\end{array}
		\\
	\end{equation}

Therefore, by \eqref{lift of foliation v(d)}, for any fixed constants
$\alpha,\beta,\mu,\nu$, we obtain
\begin{equation}\label{LAP_F-et-F-ab-v(h)}
	\begin{array}{l}
		\left(F^+_{\alpha,\beta}\right)^{v(h)}=\left\langle\alpha\frac{\partial}{\partial p^1}+\beta\frac{\partial}{\partial y^1},\ldots,\alpha\frac{\partial}{\partial p^n}+\beta\frac{\partial}{\partial y^n}\right\rangle,
		\\[0.5cm]
		\left(F^-_{\mu,\nu}\right)^{v(h)}=\left\langle \mu\frac{\partial}{\partial q^1}+\nu\frac{\partial}{\partial y^{n+1}},\ldots, \mu\frac{\partial}{\partial q^n}+\nu\frac{\partial}{\partial y^{2n}}\right\rangle .
	\end{array}
\end{equation}

\subparagraph{For each $ab\in\{cv,hv\}$, the pair
$(g^{ab},F^{ab})$ defines a para-K\"ahler structure on $TM$.\\}
Observe that from \eqref{eqlift4}, the tensors
$F^{ab}$, with $ab\in\{cv,ch,hv\}$, are involutive. Hence it remains to prove
that each $F^{ab}$ is compatible with the corresponding metric $g^{ab}$ defined by
\begin{equation}\label{LAP_g-ab}
	\left(g^{ab}\right)
	:=
	\left(F^{ab}\right)^t
	\left(\left(\omega^\flat\right)^*(d\theta)\right).
\end{equation}

Combining \eqref{LAP_g-ab} with \eqref{LAP_F-et-F-ab} and
\eqref{LAF_dtheta_TM}, one obtains
$$
g^{ab}(F^{ab}X,F^{ab}Y)
=
-g^{ab}(X,Y),
\quad
X,Y\in\mathfrak{X}(TM).
$$
Thus, the pair $(g^{ab},F^{ab})$ defines a para-K\"ahler structure on $TM$. 
\subparagraph{The pairs $\left( g^{a}, F^{a} \right)$, $a\in\{h,v\}$,
	induce  a collection of  3 webs $TM$.\\}

Let $\alpha,\beta,\mu,\nu\in\mathbb{R}$, and let
\begin{equation*}
 \left(g^{ab}_{\alpha,\beta,\mu,\nu}\right)
	:=
	\left(F^{ab}_{\alpha,\beta,\mu,\nu}\right)^t
	\left(\left(\omega^\flat\right)^*(d\theta)\right).
\end{equation*}
In a similar manner as previously, by using the fact that $\nabla^g$ is a torsion free connection, the collection $\left\{\left(g^{ab}_{\alpha,\beta,\mu,\nu},F^{ab}_{\alpha,\beta,\mu,\nu}\right) \mid \alpha\beta\mu\nu\neq 0\right\}$ forms a family of para-K\"ahler structures on $TM$.

Now, for  fixed real numbers $\alpha,\beta,\mu,\nu\in\mathbb{R}\setminus\{0\}$, let set 
		\begin{equation*}
\begin{aligned}	
&\Gamma\left(\mathcal{F}_1\right)=(F^-)^{v}\oplus(F^+)^{v},\quad
\Gamma\left(\mathcal{F}_2\right)=(F^-)^{h}\oplus(F^+)^{h},\\
&\Gamma\left({\mathcal{F}}_{\alpha,\beta,\mu,\nu}\right)=(F^-_{\alpha,\beta})^{v(h)}\oplus(F^+_{\mu,\nu})^{v(h)}.
\end{aligned}
\end{equation*}
From \eqref{LAP_F-et-F-ab} and \eqref{LAP_F-et-F-ab-v(h)} the three foliations $\mathcal{F}_i$, $i\in[3]$ are pairwise transverse. 
 Therefore, $\mathcal{W}_{\alpha,\beta,\mu,\nu}=\{\mathcal{F}_1,\mathcal{F}_2,{\mathcal{F}}_{\alpha,\beta,\mu,\nu}\}$ defines a 3 web on $TM$. That is, the collection $\left\{\mathcal{W}_{\alpha,\beta,\mu,\nu} \mid \alpha\beta\mu\nu\neq 0 \right\}$ is a family of 3-webs on $TM$.
\end{proof}

\begin{corollary}
	Let $(g,F)$ be an affine para-K\"ahler structure on $M$. By
	Lemma~\ref{LAPpushforparLem}, the vector bundle isomorphism
	$$\eta^\flat:TM\longrightarrow T^*M,
		\quad
		(x,X_x)\longmapsto \bigl(x,\eta_x(X_x,\cdot)\bigr),
		\quad
		\eta\in\{g,\omega\},$$
	transfers the family of $3$-webs on $TM$ given in
	Theorem~\ref{LAPLiftTheo} to an associated family of $3$-webs on the
	$T^*M$.
\end{corollary}

\begin{corollary}\label{LAPWithtanCoLiftRem}
Let $(g,F)$ be a para-K\"ahler structure on $M$. We denote by
$(g^{\pi},F^{\pi})$ and
$(g^{{}^*\hspace{-0.1cm}\pi},F^{{}^*\hspace{-0.1cm}\pi})$
the corresponding lifted structures on $TM$ and $T^*M$, respectively.

As a consequence of Lemma~\ref{LAPpushforparLem},
Theorem~\ref{LAPactionTheo}, and Theorem~\ref{LAPLiftTheo}, the following
assertions hold.

\begin{enumerate}
	\item The pairs
	$$
	\left(
	(\eta^\sharp)^*g^\pi,
	(F^\pi)^{\eta^\flat}
	\right)
	\quad\text{and}\quad
	\left(
	(\eta^\flat)^*g^{{}^*\hspace{-0.1cm}\pi},
	\left(F^{{}^*\hspace{-0.1cm}\pi}\right)^{\eta^\sharp}
	\right)
	$$
	define para-K\"ahler structures on $T^*M$ and $TM$, respectively.
	
	Moreover, the pairs
	$$
	\left(
	g^\pi\oplus(\eta^\sharp)^*g^\pi,
	F^\pi\oplus(F^\pi)^{\eta^\flat}
	\right)
	$$
	and
	$$
	\left(
	(\eta^\flat)^*g^{{}^*\hspace{-0.1cm}\pi}
	\oplus g^{{}^*\hspace{-0.1cm}\pi},
	\left(F^{{}^*\hspace{-0.1cm}\pi}\right)^{\eta^\sharp}
	\oplus F^{{}^*\hspace{-0.1cm}\pi}
	\right)
	$$
	define para-K\"ahler structures on the Whitney sum
	$$
	E=TM\oplus T^*M.
	$$
	
	If
	$$
	(g^{{}^*\hspace{-0.1cm}\pi},F^{{}^*\hspace{-0.1cm}\pi})
	=
	(N^{*t}g,N^{*t}F),
	\quad
	N^{*t}g=(d\theta,N^{*t}F),
	$$
	then the corresponding lifted structures are affine. On the other hand, if
	$$
	(g^{{}^*\hspace{-0.1cm}\pi},F^{{}^*\hspace{-0.1cm}\pi})
	=
	(\widetilde{g}_N,N^{*t}F),
	\quad
	\widetilde{g}_N=({}^*\hspace{-0.1cm}\pi^*\omega+d\theta,N^{*t}F),
	$$
	then the structures
	$$
	\left(
	(\eta^\flat)^*g^{{}^*\hspace{-0.1cm}\pi},
	\left(F^{{}^*\hspace{-0.1cm}\pi}\right)^{\eta^\sharp}
	\right)
	$$
	and
	$$
	\left(
	(\eta^\flat)^*g^{{}^*\hspace{-0.1cm}\pi}
	\oplus g^{{}^*\hspace{-0.1cm}\pi},
	\left(F^{{}^*\hspace{-0.1cm}\pi}\right)^{\eta^\sharp}
	\oplus F^{{}^*\hspace{-0.1cm}\pi}
	\right)
	$$
	are affine if and only if
	$(g,F)$
	is affine.
	
	\item Let $\psi\in\operatorname{Diff}(M)$. We define
	$$
	\widetilde{L}_1(\psi)=\widehat{\psi},
	\quad
	\widetilde{L}_2(\psi)
	=
	\eta^\sharp\circ\widehat{\psi}\circ\eta^\flat.
	$$
	Moreover, we set
	$$
	\widetilde{L}_1(M,g,F)
	=
	\left(
	T^*M,
	g^{{}^*\hspace{-0.1cm}\pi},
	F^{{}^*\hspace{-0.1cm}\pi}
	\right),
	$$
	and
	$$
	\widetilde{L}_2(M,g,F)
	=
	\left(
	TM,
	g^\pi,
	F^\pi
	\right).
	$$
	
	A lift of the action $\mathcal{A}$ defined in
	\eqref{LAPactMap} of Theorem~\ref{LAPactionTheo} to $T^*M$ is given by
	$$
	\mathcal{A}_{T^*M}:
	\widetilde{L}_1(\operatorname{Diff}(M))
	\times
	\widetilde{L}_1(\Pa k)
	\longrightarrow
	\widetilde{L}_1(\Pa k),
	$$
	$$
	\mathcal{A}_{T^*M}
	\left(
	\widetilde{L}_1(\psi),
	(g^{{}^*\hspace{-0.1cm}\pi},
	F^{{}^*\hspace{-0.1cm}\pi})
	\right)
	=
	\left(
	\left(\widetilde{L}_1(\psi)^{-1}\right)^*
	g^{{}^*\hspace{-0.1cm}\pi},
	\left(F^{{}^*\hspace{-0.1cm}\pi}\right)^{\widetilde{L}_1(\psi)}
	\right).
	$$
	
	Similarly, the lifted action on $TM$ is defined by
	$$
	\mathcal{A}_{TM}:
	\widetilde{L}_2(\operatorname{Diff}(M))
	\times
	\widetilde{L}_2(\Pa k)
	\longrightarrow
	\widetilde{L}_2(\Pa k),
	$$
	$$
	\mathcal{A}_{TM}
	\left(
	\widetilde{L}_2(\psi),
	(g^\pi,F^\pi)
	\right)
	=
	\left(
	\left(\widetilde{L}_2(\psi)^{-1}\right)^*g^\pi,
	(F^\pi)^{\widetilde{L}_2(\psi)}
	\right).
	$$
	
	On the Whitney sum $E=TM\oplus T^*M,$
	the induced action is the componentwise action $\mathcal{A}_{E}	=
	\mathcal{A}_{TM}\otimes\mathcal{A}_{T^*M}.$
	
	For each $i\in\{1,2\}$, the action $\mathcal{A}$ and the lift
	$\widetilde{L}_i$ commute. More precisely,
	$$
	\mathcal{A}_i
	\left(
	\widetilde{L}_i(\psi),
	\widetilde{L}_i(M,g,F)
	\right)
	=
	\widetilde{L}_i
	\left(
	\mathcal{A}(\psi,(M,g,F))
	\right),
	$$
	for every $\psi\in\operatorname{Diff}(M)$ and every
	$(g,F)\in\Pa k$, $
	\mathcal{A}_1=\mathcal{A}_{T^*M}$, $\mathcal{A}_2=\mathcal{A}_{TM}.
	$
\end{enumerate}
\end{corollary}

%




\begin{example}\label{LAP_Example}
Let $M=\mathbb{R}^2$, and let $(x,y)$, $(x,y,u,w)$, and $(x,y,s,t)$ denote the standard coordinates on $M$, $TM\approx\mathbb{R}^4$, and $T^*M\approx\mathbb{R}^4$, respectively. Consider
\begin{equation*}
	\omega=dy\wedge dx,\quad
	g=-(dx\otimes dy+dy\otimes dx),
\end{equation*}
together with the paracomplex structure defined by
\begin{equation*}
	F^+=\left\langle \frac{\partial}{\partial x}\right\rangle,
	\quad
	F^-=\left\langle \frac{\partial}{\partial y}\right\rangle.
\end{equation*}
Observe that the pair $(g,F)$ is para-K\"ahler structure on $\mathbb{R}^2$.

Let $\theta=s\,dx+t\,dy$ be the canonical Liouville one-form on $T^*M$. The relevant lifted forms are given by
\begin{equation}\label{LAP_lift_Omega}
	d\theta=ds\wedge dx+dt\wedge dy,\quad
	\pi^*\omega=dy\wedge dx,\quad
	\left(\omega^\flat\right)^*(d\theta)=dw\wedge dx-du\wedge dy.
\end{equation}
Accordingly, the lifted metrics are
\begin{equation*}
	N^{*t}g=-
	\left(
	dx\otimes ds+dy\otimes dt+ds\otimes dx+dt\otimes dy
	\right),
\end{equation*}
and
\begin{equation*}
	\widetilde{g}_N=-
	\left(
	dx\otimes dy+dy\otimes dx
	+dx\otimes ds+dy\otimes dt
	+ds\otimes dx+dt\otimes dy
	\right).
\end{equation*}

The corresponding lifted paracomplex structures are defined by
\begin{equation}\label{LAP_lift_F}
	\begin{array}{cc}
		N^{*t}F^+=
		\left\langle 
		\frac{\partial}{\partial x},
		\frac{\partial}{\partial t}
		\right\rangle,
		&
		N^{*t}F^-=
		\left\langle 
		\frac{\partial}{\partial y},
		\frac{\partial}{\partial s}
		\right\rangle,
		\\[0.4cm]
		\left(F^+\right)^v=
		\left\langle 
		\frac{\partial}{\partial u}
		\right\rangle,
		&
		\left(F^-\right)^v=
		\left\langle 
		\frac{\partial}{\partial w}
		\right\rangle,
		\\[0.4cm]
		\left(F^+\right)^h=
		\left\langle 
		\frac{\partial}{\partial x}
		\right\rangle,
		&
		\left(F^-\right)^h=
		\left\langle 
		\frac{\partial}{\partial y}
		\right\rangle,
		\\[0.4cm]
		\left(F^+\right)^c=
		\left\langle 
		\frac{\partial}{\partial x}
		\right\rangle,
		&
		\left(F^-\right)^c=
		\left\langle 
		\frac{\partial}{\partial y}
		\right\rangle,
		\\[0.4cm]
		\left(F^+_{1,6}\right)^{v(h)}=
		\left\langle \frac{\partial}{\partial x}+
		6\frac{\partial}{\partial u}
		\right\rangle,
		&
		\left(F^-_{1,1}\right)^{v(h)}=
		\left\langle \frac{\partial}{\partial y}+
		\frac{\partial}{\partial w}
		\right\rangle,
	\end{array}
\end{equation}
where the horizontal lift is taken with respect to the Levi-Civita
connection $\nabla^g$.

In order to obtain a simple finite-dimensional illustration, we restrict the action $\A$ to the special linear group
\begin{equation*}
	\operatorname{SL}_2(\mathbb{R})
	=
	\left\{
	\psi_A:\mathbb{R}^2\longrightarrow \mathbb{R}^2,\quad
	\psi_A(x,y)=A(x,y)^t,\quad
	A\in M_2(\mathbb{R}),\quad
	\det A=1
	\right\}.
\end{equation*}
Let
\begin{equation*}
	A=
	\begin{pmatrix}
		\alpha & \beta\\
		\gamma & \delta
	\end{pmatrix},
	\quad
	\alpha\delta-\beta\gamma=1,
	\quad
	\psi=\psi_A.
\end{equation*}
Since $\det A:=\alpha\delta-\beta\gamma=1$, every element of $\operatorname{SL}_2(\mathbb{R})$ preserves the symplectic form $\omega$, and preserve $g$ if and only if $\alpha\gamma=0=\delta\beta$.

The induced action on para-K\"{a}hler structures is given by
\begin{equation}\label{LAP_lift_Act_Ex}
	g^\psi=(\psi^{-1})^*g,\quad
	\left(F^\psi\right)^+=
	\left\langle
	\alpha\frac{\partial}{\partial x}
	+
	\gamma\frac{\partial}{\partial y}
	\right\rangle,
	\quad
	\left(F^\psi\right)^-=
	\left\langle
	\beta\frac{\partial}{\partial x}
	+
	\delta\frac{\partial}{\partial y}
	\right\rangle.
\end{equation}
Observe that,
\begin{equation*}
	\psi^{-1}(x,y)
	=
	A^{-1}(x,y)^t
	=
	(\delta x-\beta y,\,-\gamma x+\alpha y).
\end{equation*}

Then the cotangent lift of $\psi$ is
\begin{equation}\label{LAP_lift_psi}
	\widehat{\psi}(x,y,s,t)
	=
	\left(
	\alpha x+\beta y,\,
	\gamma x+\delta y,\,
	\delta s-\gamma t,\,
	-\beta s+\alpha t
	\right).
\end{equation}
Its differential is represented by the block matrix
\begin{equation*}
	\widehat{\psi}_*
	=
	\begin{pmatrix}
		A & 0\\
		0 & A^{-1}
	\end{pmatrix}.
\end{equation*}
Since $\widehat{\psi}$ is the cotangent lift of $\psi$, it preserves the canonical one-form $\theta$, and hence it preserves the canonical symplectic form $d\theta$. Moreover, the cotangent lift acts naturally on the base and dual variables, so it also preserves the canonical neutral metric $N^{*t}g$. Therefore,
\begin{equation*}
	\widehat{\psi}^{\,*}d\theta=d\theta,
	\quad
	\widehat{\psi}^{\,*}N^{*t}g=N^{*t}g.
\end{equation*}
Combining \eqref{LAP_lift_F}, \eqref{LAP_lift_Omega}, \eqref{LAP_lift_Act_Ex}, and \eqref{LAP_lift_psi}, the lifted action $\mathcal{A}|_{\operatorname{SL}_2}$ may be described in the different equivalent ways appearing in Corollary~\ref{LAPWithtanCoLiftRem}. For instance,
\begin{equation*}
	\widehat{\psi}^{\,*}N^{*t}g=N^{*t}g,
\end{equation*}
and
\begin{equation*}
	\left(\left(N^{*t}F\right)^{\widehat{\psi}}\right)^+
	=
	\left\langle
	\widehat{\psi}_*
	\frac{\partial}{\partial x},
	\widehat{\psi}_*
	\frac{\partial}{\partial t}
	\right\rangle,
	\quad
	\left(\left(N^{*t}F\right)^{\widehat{\psi}}\right)^-
	=
	\left\langle
	\widehat{\psi}_*
	\frac{\partial}{\partial y},
	\widehat{\psi}_*
	\frac{\partial}{\partial s}
	\right\rangle.
\end{equation*}
Equivalently,
\begin{equation*}
	\left(\left(N^{*t}F\right)^{\widehat{\psi}}\right)^+
	=
	\left\langle
	\alpha\frac{\partial}{\partial x}
	+
	\gamma\frac{\partial}{\partial y},
	-\gamma\frac{\partial}{\partial s}
	+
	\alpha\frac{\partial}{\partial t}
	\right\rangle,
\end{equation*}
and
\begin{equation*}
	\left(\left(N^{*t}F\right)^{\widehat{\psi}}\right)^-
	=
	\left\langle
	\beta\frac{\partial}{\partial x}
	+
	\delta\frac{\partial}{\partial y},
	\delta\frac{\partial}{\partial s}
	-
	\beta\frac{\partial}{\partial t}
	\right\rangle.
\end{equation*}
\end{example}

These two-dimensional examples illustrate the lifting procedure in the simplest
nontrivial setting and provide useful models for understanding the corresponding
constructions in higher dimensions and in more general geometric contexts.



\small

\begin{thebibliography}{99}
	\bibitem{Dom} P. Dombrowski.
	On the Geometry of the Tangent Bundle.
	{\it Journal f\"ur die reine und angewandte Mathematik} \textbf{210} (1962) 73--88.
	
	\bibitem{1} F. Etayo, R. Santamaria and U. Tr\'{\i}as.
	The geometry of a bi-Lagrangian manifold.
	{\it Differential Geometry and Applications} \textbf{24}\,(1) (2006) 33--59.
	
	\bibitem{FAT} A. Ferrag, A. Aoufi and M. Tadjine.
	Existence and Uniqueness of Solution for Stochastic Nonlocal Random Functional Integral Equation.
	{\it Nonlinear Dynamics and Systems Theory} \textbf{25}\,(1) (2025) 53--58.
	

	\bibitem{Hess} H. Hess.
	Connections on symplectic manifolds and geometric quantization.
	In: {\it Differential Geometric Methods in Mathematical Physics}.
	(P.\,L. Garc{\'{\i}}a, A. P{\'e}rez-Rend{\'o}n and J.\,M. Souriau).
	Springer-Verlag, Berlin, 1980, 153--166.
	
	\bibitem{TNB2} B. Tangue Ndawa.
	Infinite Lifting of an Action of Symplectomorphism Group on the set of bi-Lagrangian structures.
	{\it Journal of Geometry and Mechanics} \textbf{14}\,(3) (2022) 409--426.
		\bibitem{TNB4} B. Tangue Ndawa.
	Dynamics on Bi-Lagrangian Structures and Cherry maps.
	{\it Arxiv}   	arXiv:2508.12350.
	\bibitem{TNB5} B. Tangue Ndawa, F. Ngakeu and N. Saipele Nansidi.
	Bi-Lagrangian and transverse Dirac structures induced on the Whitney sum.
	{\it Arxiv}  	arXiv:2607.11804.
	
	\bibitem{YKI} K. Yano and S. Kobayashi.
	Prolongations of tensor fields and connections to tangent bundles I.
	{\it Journal of the Mathematical Society of Japan} \textbf{18}\,(2) (1966) 195--210.
	
	\bibitem{YKII} K. Yano and S. Kobayashi.
	Prolongations of tensor fields and connections to tangent bundles II.
	{\it Journal of the Mathematical Society of Japan} \textbf{18}\,(3) (1966) 236--246.
	
\end{thebibliography}
\end{document}